\documentclass{article}

\usepackage{graphicx}
\usepackage{indentfirst}
\usepackage{amsmath,amsfonts,amsthm,amssymb}
\usepackage{mathrsfs}
\usepackage{amscd}
\usepackage{hyperref}

\def\co{\colon\thinspace}
\DeclareMathAlphabet{\mathsfsl}{OT1}{cmss}{m}{sl}

\newcommand{\spin}{\mathrm{Spin}^c}
\newcommand{\relspin}{\underline{\mathrm{Spin}^c}}

\newtheorem{thm}{Theorem}[section]
\newtheorem{lem}[thm]{Lemma}
\newtheorem{cor}[thm]{Corollary}
\newtheorem{prop}[thm]{Proposition}

\theoremstyle{definition}

\newtheorem{rem}[thm]{Remark}

\begin{document}

\title{Heegaard Floer correction terms and rational genus bounds}

\author{{\Large Yi NI\footnote{Corresponding author}\quad Zhongtao WU\footnote{Present address: Department of Mathematics, The Chinese Universiy of Hong Kong,
Shatin, Hong Kong,
{\it Email :}
ztwu@math.cuhk.edu.hk}}\\{\normalsize Department of Mathematics, Caltech, MC 253-37}\\
{\normalsize 1200 E California Blvd, Pasadena, CA
91125}\\{\small\it Email\/:\quad\rm yini@caltech.edu\quad zhongtao@caltech.edu}}

\date{}
\maketitle

\begin{abstract}
Given an element in the first homology of a rational homology $3$--sphere $Y$, one can consider the minimal rational genus of all knots in this homology class.
This defines a function $\Theta$ on $H_1(Y;\mathbb Z)$, which was introduced by Turaev as an analogue of Thurston norm. We will give a lower bound for this function using the correction terms in Heegaard Floer homology. As a corollary, we show that Floer simple knots in L-spaces are genus minimizers in their homology classes, hence answer questions of Turaev and Rasmussen about genus minimizers in lens spaces.
\end{abstract}

\section{Introduction}

Heegaard Floer homology, introduced by Ozsv\'ath and Szab\'o \cite{OSzAnn1}, has been very successful in the study of low-dimensional topology. One important feature of Heegaard Floer homology which makes it so useful is that it gives a lower bound for the genus of surfaces in a given homology class. In dimension $3$, it determines the Thurston norm \cite{OSzGenus}. In dimension $4$, the adjunction inequality \cite{OSz4Manifold} gives a lower bound to the genus of surfaces which is often sharp, and the concordance invariant \cite{OSz4Genus} gives a lower bound to the slice genus of knots. Such kind of results have been known before in Donaldson theory and Seiberg--Witten theory \cite{KMMilnor, KMDonaldson,KMThom,MSzT,KMNorm}, but the combinatorial nature of Heegaard Floer homology makes the corresponding results easier to use in many problems.

In this paper, we will study a new type of genus bounds. Suppose that $Y$ is a closed oriented $3$--manifold, there is a kind of ``norm'' function one can define on the torsion subgroup of $H_1(Y;\mathbb Z)$. To define it, let us first recall  the rational genus of a rationally null-homologous knot $K\subset Y$ defined by Calegari and Gordon \cite{CG}.

Suppose that $K$ is a rationally null-homologous oriented knot in $Y$, and $\nu(K)$ is a tubular neighborhood of $K$. A properly embedded oriented connected surface $F\subset Y\backslash\overset{\circ}{\nu}(K)$ is called a {\it rational Seifert surface} for $K$, if $\partial F$ consists of coherently oriented parallel curves on $\partial\nu(K)$, and the orientation of $\partial F$ is coherent with the orientation of $K$. The {\it rational genus} of $K$ is defined to be
$$||K||=\min_{F}\frac{\max\{0,-\chi(F)\}}{2|[\mu]\cdot[\partial F]|},$$
where $F$ runs over all the rational Seifert surfaces for $K$, and $\mu\subset\partial\nu(K)$ is the meridian of $K$.

The rational genus is a natural generalization of the genus of null-homologous knots. Moreover, given a torsion class in $H_1(Y)$, one can consider the minimal rational genus for all knots in this torsion class. More precisely,
given $a\in\mathrm{Tors} H_1(Y)$, let 
$$\Theta(a)=\min_{K\subset Y,\:[K]=a}2||K||.$$
This $\Theta$ was introduced by Turaev \cite{TuFunc} in a slightly different form. Turaev regarded $\Theta$ as an analogue of Thurston norm \cite{Th}, in the sense that it measures the minimal normalized Euler characteristic of a ``folded surface'' representing a given class in $H_2(Y;\mathbb Q/\mathbb Z)$.

In \cite{TuFunc}, Turaev gave a lower bound for $\Theta$ in terms of his torsion function. When $b_1(Y)>0$, Turaev's torsion function is a kind of Euler characteristic of Heegaard Floer homology \cite{OSzAnn2}, so his lower bound can be reformulated in terms of Heegaard Floer homology. (One can even expect to get a better bound with Heegaard Floer homology.) When $b_1(Y)=0$, the relationship between Turaev's torsion function and Heegaard Floer homology is not very clear in the literature. Nevertheless, our following theorem gives an independent  lower bound in terms of the correction terms $d(Y,\mathfrak s)$ in Heegaard Floer homology \cite{OSzAbGr}.

\begin{thm}\label{thm:GenusBound}
Suppose that $Y$ is a rational homology $3$--sphere, $K\subset Y$ is a knot, $F$ is a rational Seifert surface for $K$. Then
\begin{equation}\label{eq:GenusBound}
1+\frac{-\chi(F)}{|[\partial F]\cdot [\mu]|}\ge \max_{\mathfrak s\in\spin(Y)}\big\{d(Y,\mathfrak s+\mathrm{PD}[K])-d(Y,\mathfrak s)\big\}.
\end{equation}
The right hand side of (\ref{eq:GenusBound}) only depends on the manifold $Y$ and the homology class of $K$, so it gives a lower bound for $1+\Theta(a)$ for the homology class $a=[K]$.
\end{thm}

We find Theorem~\ref{thm:GenusBound} quite interesting because it unveils some topological information contained in Heegaard Floer homology of rational homology spheres. Such information is relatively rare in the literature comparing to the case of manifolds with positive $b_1$, where one can get useful information like Thurston norm \cite{OSzGenus} and fibration \cite{NiClosedF}.

Theorem~\ref{thm:GenusBound} is particularly useful when $Y$ is an L-space and the homology class contains a Floer simple knot.
Recall that a rational homology $3$--sphere $Y$ is an {\it L-space} if $\mathrm{rank}\widehat{HF}(Y)=|H_1(Y;\mathbb Z)|$. A rationally null-homologous knot $K$ in a $3$--manifold $Y$ is {\it Floer simple} if 
$\mathrm{rank}\widehat{HFK}(Y,K)=\mathrm{rank}\widehat{HF}(Y)$.

\begin{thm}\label{thm:SimpleL}
Suppose that $Y$ is an L-space, $K$ is a Floer simple knot in $Y$. If $K_1$ is another knot in $Y$ with $[K_1]=[K]\in H_1(Y;\mathbb Z)$, then
$$||K||\le ||K_1||.$$
\end{thm}

An important class of L-spaces is lens spaces.
The question of computing $\Theta$ for lens spaces was first considered by Turaev \cite{TuFunc}. The lower bound given by Turaev (for any manifold) is always less than $1$, but the value of $\Theta$ can be much larger than $1$ even for lens spaces. For example, if $a\sim \frac p2$, using (\ref{eq:Lp1}) in Section~\ref{sect:Appl} we get a lower bound $\sim\frac p4$ for $\Theta(a)$ in $L(p,1)$. So Turaev's bound is not sharp for lens spaces. 

Hedden \cite{HedBerge} and Rasmussen \cite{RasBerge} observed that for any $1$--dimensional homology class in a lens space, there exists a knot in this homology class which is Floer simple. Let $U_0\cup U_1$ be a genus $1$ Heegaard splitting of a lens space $L(p,q)$, and let $D_0,D_1$ be meridian disks in $U_0,U_1$ such that $\partial D_0\cap\partial D_1$ consists of exactly $p$ points. A knot in $L(p,q)$ is called {\it simple} if it is either the unknot or the union of two arcs $a_0\subset D_0$ and $a_1\subset D_1$. Up to isotopy there is exactly one simple knot in each homology class in $H_1(L(p,q))$, and every simple knot is Floer simple.

Rasmussen \cite{RasBerge} conjectured that simple knots are genus minimizers in their homology classes, and verified this conjecture for dual Berge knots. In fact, he proved  that primitive knots in L-spaces with rational genus less than $\frac12$ are genus minimizing. As a consequence of our Theorem~\ref{thm:SimpleL}, we verify Rasmussen's conjecture in general. 

\begin{cor}
Simple knots in lens spaces are genus minimizers in their homology classes. 
\end{cor}

\begin{rem}
Rasmussen also conjectured that simple knots are the {\bf unique} genus minimizers in their homology classes. As pointed out in \cite{HedBerge,RasBerge}, the uniqueness of genus minimizers in lens spaces would imply the Berge Conjecture on lens space surgeries \cite{Berge}.
\end{rem}

Unlike the Heegaard Floer bound for Thurston norm, our bound for $\Theta$ is not always sharp. For example, suppose $K\subset Y$ is a knot in a homology sphere such that the half degree of its Alexander polynomial is equal to its genus $g$. Let $p\ge 4g-2$ be an integer, $Y_p(K)$ be the manifold obtained by $p$--surgery on $K$, and $K'\subset Y_{p}(K)$ be the dual knot of the surgery. Then the lower bound given by Turaev on $\Theta([K'])$ in $Y_p(K)$ is $\frac{2g-1}p$, and can be obviously realized  \cite[Section~6.2]{TuFunc}. In this case, the bound given by Theorem~\ref{thm:GenusBound} is not always sharp (see Section~\ref{sect:Appl}). Nevertheless, using Heegaard Floer homology we will prove the following result.

\begin{prop}\label{prop:S3Surg}
Suppose that $K\subset S^3$ is a knot with genus $g$, and that $p\ge2g$ is an integer, then the dual knot $K'\subset S^3_p(K)$ is a genus minimizer in its homology class. Namely, $\Theta([K'])=\frac{2g-1}p$.
\end{prop}

This paper is organized as follows. In Section~\ref{sect:prelim} we will give the necessary background on Heegaard Floer homology. We will focus on the construction of the knot Floer homology of rationally null-homologous knots. In Section~\ref{sect:Symm}, we prove a symmetry relation in knot Floer homology. In Section~\ref{sect:Bound}, we prove Theorem~\ref{thm:GenusBound} by analyzing the knot Floer chain complex of rationally null-homologous knots. In Section~\ref{sect:Appl}, we apply Theorem~\ref{thm:GenusBound} to some examples, thus prove Theorem~\ref{thm:SimpleL} and Proposition~\ref{prop:S3Surg}.

\vspace{5pt}\noindent{\bf Acknowledgements.}\quad The first author wishes to thank Jacob Rasmussen for asking the question which motivated this work. The first author was
partially supported by an AIM Five-Year Fellowship, NSF grant
number DMS-1103976 and an Alfred P. Sloan Research Fellowship. The second author was supported by a Simons Postdoctoral Fellowship.


\section{Preliminaries}\label{sect:prelim}

\subsection{Correction terms in Heegaard Floer homology}

Heegaard Floer homology, introduced by
Ozsv\'ath and Szab\'o \cite{OSzAnn1}, is an invariant  for closed oriented Spin$^c$ $3$--manifolds $(Y,\mathfrak s)$,
  taking the form of a collection of related homology groups as  $\widehat{HF}(Y,\mathfrak s)$, $HF^{\pm}(Y,\mathfrak s)$, and $HF^\infty(Y,\mathfrak s)$.
There is a $U$--action on Heegaard Floer homology groups.
When $\mathfrak s$ is torsion, there is an absolute Maslov $\mathbb Q$--grading on the Heegaard Floer homology groups. The $U$--action decreases the grading by $2$.

For a rational homology $3$--sphere $Y$ with a Spin$^c$ structure $\mathfrak s$, $HF^+(Y,\mathfrak s)$ can be decomposed as the direct sum of two groups: the first group is the image of $HF^\infty(Y,\mathfrak s)\cong\mathbb Z[U,U^{-1}]$ in $HF^+(Y,\mathfrak s)$,
which is isomorphic to $\mathcal T^+=\mathbb Z[U,U^{-1}]/UZ[U]$, and its minimal absolute  $\mathbb{Q}$--grading is an invariant of $(Y,\mathfrak s)$, denoted by $d(Y,\mathfrak s)$, the {\it correction term} \cite{OSzAbGr}; the second group is the quotient modulo the above image and is denoted by $HF_{\mathrm{red}}(Y,\mathfrak s)$.  Altogether, we have $$HF^+(Y,\mathfrak s)=\mathcal{T}^+\oplus HF_{\mathrm{red}}(Y,\mathfrak s).$$

The correction term satisfies
\begin{equation}\label{eq:CorrSymm}
d(Y,\mathfrak s)=d(Y,J\mathfrak s),\qquad d(-Y,\mathfrak s)=-d(Y,\mathfrak s),
\end{equation}
where $J\co \spin(Y)\to\spin(Y)$ is the conjugation on $\spin(Y)$, and $-Y$ is $Y$ with the orientation reversed.

\subsection{Relative {\rm Spin$^c$} structures}

Let $M$ be a compact 3--manifold with boundary consisting of tori.
Let $v_1$ and $v_2$ be two nowhere vanishing
vector fields on $M$, whose restriction on each component of
$\partial M$ is the outward normal vector field.
We say $v_1$ and $v_2$ are {\it homologous}, if they are homotopic
in the complement of a ball in $M$, and the homotopy is through
nowhere vanishing vector fields which restrict to the outward normal vector field on $\partial M$. The homology classes of such vector fields
are called {\it relative {\rm Spin}$^c$ structures} on $M$, and
the set of all relative Spin$^c$ structures is denoted by
$\relspin(M,\partial M)$, which is an affine
space over $H^2(M,\partial M)$.

When $K$ is an oriented knot in a closed oriented 3--manifold $Y$,
let $M=Y\backslash\overset{\circ}{\nu}(K)$. Then we also denote
$\relspin(M,\partial M)$ by $\relspin(Y,K)$.

\begin{rem}
There are several different conventions in the literature for the boundary condition of vector fields representing a relative Spin$^c$ structure. In \cite{OSzLink,NiLinkNorm}, the restriction of the vector fields on the boundary is	tangent to the boundary. Our treatment in this paper is the one taken in \cite{OSzRatSurg}.
\end{rem}

Suppose $K$ is an oriented rationally null-homologous knot in a closed manifold $Y^3$,
$(\Sigma,\mbox{\boldmath${\alpha}$},\mbox{\boldmath${\beta}$},
w, z)$ is a doubly-pointed Heegaard
diagram associated to the pair $(Y,K)$. There is a map
$$\underline{\mathfrak s}_{w,z}\co\mathbb T_{\alpha}\cap\mathbb T_{\beta}\to\relspin(Y,K),$$
defined in \cite{OSzRatSurg}. We sketch the definition of
$\underline{\mathfrak s}_{w,z}$ as follows.

Let $f\co Y\to[0,3]$ be a Morse function corresponding to the
Heegaard diagram, $\nabla f$ is the gradient vector field
associated to $f$. Let $\gamma_{w}$ be the 
flowline of $\nabla f$ passing through $w$, which connects
the index-zero critical point to the index-three critical point.
Similarly, define $\gamma_{z}$. Suppose $\mathbf
x\in\mathbb T_{\alpha}\cap\mathbb T_{\beta}$, then
$\mbox{\boldmath${\gamma}$}_{\mathbf x}$ denotes the union of the
flowlines connecting index-one critical points to index-two
critical points, and passing through the points in $\mathbf x$.

We construct a nowhere vanishing vector field $v$. Outside a
neighborhood $\nu(\gamma_{w}\cup\gamma_{z}\cup\mbox{\boldmath${\gamma}$}_{\mathbf x})$, $v$ is identical
with $\nabla f$. Then one can extend $v$ over the balls
$\nu(\mbox{\boldmath${\gamma}$}_{\mathbf x})$. We can also
extend $v$ over $\nu(\gamma_{w}\cup\gamma_{
z})$, so that the closed orbits of $v$, which pass through $w$ and $z$, give the oriented knot
$K=\gamma_{z }-\gamma_{w}$. There may be many
different choices to extend $v$ over $\nu(\gamma_{
w}\cup\gamma_{z})$, we choose the extension as in
\cite[Subsection~2.4]{OSzRatSurg}.

Now we let $\underline{\mathfrak s}_{w,z}(\mathbf
x)$ be the relative Spin$^c$ structure given by
$v|_{Y\backslash\overset{\circ}{\nu}(K)}$. It is easy to check that
$\underline{\mathfrak s}_{w,z}$ is a well-defined
map.

Let $u$ be a vector field on $S^1\times D^2$ as described in \cite[Subsection~2.2]{OSzRatSurg}. More precisely, $u$ is the inward normal vector field on the boundary torus, $u$ is transverse to the meridian disks in the interior of $S^1\times D^2$,  and the core of $S^1\times D^2$ is a closed orbit of $u$. Given $\xi\in\relspin(Y,K)$, 
let $v$ be a vector field representing $\xi$, then we can glue $v$ and $u$ together to get a vector field on $Y$, which represents a Spin$^c$ structure on $Y$. Hence we get a map 
$$G_{Y,K}\co \relspin(Y,K)\to\mathrm{Spin}^c(Y).$$
We call $G_{Y,K}(\xi)$ the {\it underlying} Spin$^c$ structure of $\xi$. It is shown in \cite{OSzRatSurg} that
$$G_{Y,K}(\underline{\mathfrak s}_{w,z}(\mathbf x))=\mathfrak s_w(\mathbf x).$$

\subsection{Knot Floer homology of rationally null-homologous knots}

Suppose that $K$ is a rationally null-homologous knot in a closed $3$--manifold
$Y$. Let $$(\Sigma,\mbox{\boldmath${\alpha}$},
\mbox{\boldmath$\beta$},w,z)$$ be a doubly pointed Heegaard
diagram for $(Y,K)$. Fix a Spin$^c$ structure $\mathfrak s$ on $Y$ and let
$\xi\in\underline{\mathrm{Spin}}^c(Y,K)$ be a relative Spin$^c$
structure whose underlying Spin$^c$ structure is $\mathfrak s$. Let $CFK^{\infty}(Y,K,\xi)$ be an
abelian group freely generated by triples $[\mathbf x,i,j]$ with
$$\mathbf x\in\mathbb T_{\alpha}\cap\mathbb T_{\beta},\quad\mathfrak
s_w(\mathbf x)=\mathfrak s$$ and 
\begin{equation}\label{eq:Generator}
\underline{\mathfrak
s}_{w,z}(\mathbf x)+(i-j)PD[\mu]=\xi.
\end{equation}
The chain complex is endowed
with the differential
$$\partial^{\infty}[\mathbf x,i,j]=\sum_{\mathbf y\in\mathbb T_{\alpha}\cap\mathbb T_{\beta}}
\sum_{\{\phi\in\pi_2(\mathbf x,\mathbf
y)|\mu(\phi)=1\}}\#(\widehat{\mathcal M}(\phi))[\mathbf
y,i-n_w(\phi),j-n_z(\phi)].$$ The homology of
$(CFK^{\infty}(Y,K,\xi),\partial^{\infty})$ is denoted
$HFK^{\infty}(Y,K,\xi)$.

The grading $j$ gives a filtration on
$CFK^{0,*}(Y,K,\xi)$, the associated graded complex is denoted
$\widehat{CFK}(Y,K,\xi)$.

Given a knot $K$ in a rational homology sphere $Y$, let $F$ be a rational Seifert surface for $K$, then
there is an affine map $A\co \relspin(Y,K)\to\mathbb Q$ satisfying
\begin{equation}\label{eq:A}
A(\xi_1)-A(\xi_2)=\frac{\langle \xi_2-\xi_1,[F]\rangle}{|[\partial F]\cdot[\mu]|}.
\end{equation}
This map can be defined and determined by (\ref{eq:A}), once we fix the value of $A$ at a $\xi_0\in\relspin(Y,K)$ .

Let $$\mathcal B_{Y,K}=\left\{\xi\in\relspin(Y,K)\left|\:\widehat{HFK}(Y,K,\xi)\ne0\right.\right\}.$$ Let $$A_{\max}=\max\{A(\xi)|\:\xi\in\mathcal B_{Y,K}\}, \quad A_{\min}=\min\{A(\xi)|\:\xi\in\mathcal B_{Y,K}\}.$$
We reformulate \cite[Theorem~1.1]{NiLinkNorm} for knots as follows.

\begin{thm}\label{thm:RatGenus}
Suppose $K$ is a knot in a rational homology sphere $Y$, $F$ is a minimal genus rational Seifert surface for $K$, then
$$\frac{-\chi(F)+|[\partial F]\cdot[\mu]|}{|[\partial F]\cdot[\mu]|}=A_{\max}-A_{\min}.$$
\end{thm}

Suppose $K\subset Y$ is a rationally null-homologous knot. We say that $K$ is {\it rationally fibered}, if the complement of $K$ is a surface bundle over $S^1$, and the fiber is a rational Seifert surface for $K$. 

\begin{thm}\label{thm:RatFiber}
Suppose $K\subset Y$ is a rationally null-homologous knot, $F$ is a rational Seifert surface for $K$. Then the complement of $K$ fibers over $S^1$ with fiber $F$ if and only if the group
$$\bigoplus_{\xi\in\relspin(Y,K), A(\xi)=A_{\max}}\widehat{HFK}(Y,K,\xi)$$
is isomorphic to $\mathbb Z$.
\end{thm}
\begin{proof}
This follows from \cite{NiFibred}, \cite[Proposition~5.15]{NiLinkNorm}, and the fact that a knot is rationally fibered if and only if any of its cables is rationally fibered.
\end{proof}

\subsection{Rational domains and the relative rational bigrading}

When $\mathfrak s$ is a torsion Spin$^c$ structure over $Y$, as in
Ozsv\'ath--Szab\'o \cite{OSzKnot} there is an absolute $\mathbb
Q$--grading on $CFK^{\infty}(Y,K,\xi)$ and the induced complexes.
Let $\widehat{HFK}_d(Y,K,\xi)$ be the summand of
$\widehat{HFK}(Y,K,\xi)$ at the absolute grading $d$.

We first recall Lee--Lipshitz's construction of the relative $\mathbb Q$--grading \cite{LL}. Suppose $D_1,\dots,D_N$ are closures of the components of $\Sigma-\mbox{\boldmath${\alpha}$}-
\mbox{\boldmath$\beta$}$, thought of as $2$--chains. Suppose $\psi=\sum_ia_iD_i$ for some rational numbers $a_i$, and let $\partial_{\alpha}\psi$ be the intersection of $\partial \psi$ with $\mbox{\boldmath${\alpha}$}$, then $\partial\partial_{\alpha}\psi$ is a rational linear combination of intersection points between $\alpha$ and $\beta$ curves. We say $\psi$ is a {\it rational domain} connecting $\mathbf x=(x_1,\dots,x_g)\in\mathbb T_{\alpha}\cap\mathbb T_{\beta}$ to $\mathbf y=(y_1,\dots,y_g)\in\mathbb T_{\alpha}\cap\mathbb T_{\beta}$, if $$\partial\partial_{\alpha}\psi=y_1+\cdots+y_g-(x_1+\cdots+x_g).$$ If $\psi=\sum_ia_iD_i$ is a  rational domain connecting $\mathbf x$ to $\mathbf y$ with $n_w(\psi)=0$, then we define the Maslov index \footnote{Unfortunately, we use $\mu$ to denote both the Maslov index of a rational domain and the meridian of a knot. This should not cause confusion in our current paper.}
$$\mu(\psi)=\sum_ia_i\big(e(D_i)+n_{\mathbf x}(D_i)+n_{\mathbf y}(D_i)\big),$$
where $e(D_i)$ is the Euler measure of $D_i$ as defined by Lipshitz \cite{Lip}.

The following lemma is contained in the last paragraph of \cite[Section~2]{LL}.

\begin{lem}[Lee--Lipshitz]\label{lem:LL}
Suppose $\psi$ is a rational domain connecting $\mathbf x$ to $\mathbf y$ with $n_w(\psi)=0$, then 
$$\mathrm{Gr}(\mathbf x)-\mathrm{Gr}(\mathbf y)=\mu(\psi).$$
\end{lem}

There is a similar formula for the relative Alexander grading.

\begin{lem}\label{lem:RelAlex}
Suppose $\psi$ is a rational domain connecting $\mathbf x$ to $\mathbf y$, then
$$A(\mathbf x)-A(\mathbf y)=n_{z}(\psi)-n_{w}(\psi).$$
\end{lem}
\begin{proof}
Let $F$ be a rational Seifert surface for $K$.
By (\ref{eq:A}) and \cite[Lemma~2.19]{OSzAnn1}, \begin{eqnarray*}
A(\mathbf x)-A(\mathbf y)&=&\frac{\langle\mathrm{PD}[\partial \psi],[F]\rangle}{|[\partial F]\cdot[\mu]|}\\
&=&\frac{[\partial\psi]\cdot[F]}{|[\partial F]\cdot[\mu]|}
\end{eqnarray*}
which is the rational linking number between $\partial\psi$ and $K$. This linking number can also be computed by
$[\psi]\cdot[K]=n_{z}(\psi)-n_{w}(\psi)$.
\end{proof}


\section{Symmetries in knot Floer homology}\label{sect:Symm}

Suppose that $K$ is a rationally null-homologous knot in a $3$--manifold $Y$.
Let $$\Gamma_1=(\Sigma,\mbox{\boldmath${\alpha}$},
\mbox{\boldmath$\beta$},w,z)$$ be a doubly pointed Heegaard
diagram for $(Y,K)$. Then
$$\Gamma_2=(-\Sigma,\mbox{\boldmath$\beta$},\mbox{\boldmath${\alpha}$},
z,w)$$ is also a Heegaard diagram for $(Y,K)$. We call $\Gamma_2$ the {\it dual diagram} of $\Gamma_1$. Let $\underline{\mathfrak s}^i_{w,z}(\mathbf x)$ be the associated relative Spin$^c$ structure for the diagram $\Gamma_i$.
We define a map $\widetilde J\co \mathcal B_{Y,K}\to \relspin(Y,K)$ as follows. If $$\underline{\mathfrak s}^1_{w,z}(\mathbf x)=\xi,$$ for some $\mathbf x$,
then define $$\widetilde J\xi=\widetilde J_{\Gamma_1}\xi=\underline{\mathfrak s}^2_{z,w}(\mathbf x).$$

\begin{lem}\label{lem:JInvariance}
The map $\widetilde J$ does not depend on the diagram $\Gamma_1$.
\end{lem}
\begin{proof}
This follows from the standard procedure of proving the invariance of $\widehat{HFK}(Y,K,\xi)$. Suppose $\Gamma_1,\Gamma_1'$ are two different diagrams for $(Y,K)$, then they are connected by the following types of moves:\newline
$\bullet$ isotopies of the $\mbox{\boldmath${\alpha}$}$ and the $\mbox{\boldmath${\beta}$}$ without crossing $w,z$,\newline
$\bullet$ handleslides amongst the $\mbox{\boldmath${\alpha}$}$ or the $\mbox{\boldmath${\beta}$}$,\newline
$\bullet$ stabilizations.\newline
Then the dual diagrams $\Gamma_2$ and $\Gamma_2'$ are also related by the corresponding moves. Tracing these moves, the proof of the invariance of $\widehat{HFK}(Y,K,\xi)$ implies that  $\widetilde J_{\Gamma_1}=\widetilde J_{\Gamma_1'}$.
\end{proof}

\begin{lem}\label{lem:Js+K}
Suppose $\xi\in\mathcal B_{Y,K}$, then
$$G_{Y,K}(\widetilde J\xi)=JG_{Y,K}(\xi)+\mathrm{PD}[K].$$
\end{lem}
\begin{proof}
Suppose $\xi=\underline{\mathfrak s}^1_{w,z}(\mathbf x)$, then $G_{Y,K}(\xi)=\mathfrak s^1_w(\mathbf x)$, hence $JG_{Y,K}(\xi)=\mathfrak s^2_w(\mathbf x)$.
On the other hand, $\widetilde J\xi=\underline{\mathfrak s}^2_{z,w}(\mathbf x)$, so $G_{Y,K}(\widetilde J\xi)=\mathfrak s^2_{z}(\mathbf x)$.
By \cite[Lemma~2.19]{OSzAnn1} or \cite[Equation~(1)]{RasBerge}, $$\mathfrak s^2_z(\mathbf x)=\mathfrak s^2_w(\mathbf x)+\mathrm{PD}[K].$$
So our conclusion holds.
\end{proof}

The following theorem is an analogue of \cite[Proposition~3.10]{OSzKnot}. We have been informed that some cases of the theorem are already contained in \cite{BGH}.

\begin{thm}\label{thm:RatSymm}
Let $\mathfrak s$ be a Spin$^c$ structure over $Y$, and
let $\xi\in\underline{\mathrm{Spin}^c}(Y,K)$ be a relative
Spin$^c$ structure with underlying Spin$^c$ structure $\mathfrak s$. 

\noindent(a) There
is an isomorphism of chain complexes
$$\widehat{CFK}(Y,K,\xi)\cong\widehat{CFK}(Y,K,\widetilde J\xi).$$

\noindent(b) 
The map $\widetilde J$ maps $\mathcal B_{Y,K}$ into $\mathcal B_{Y,K}$, and $\widetilde J^2=\mathrm{id}$. 

\noindent(c) If $\mathfrak s$ is a torsion Spin$^c$ structure, then there is an isomorphism of absolutely graded chain complexes:
$$\widehat{CFK}_*(Y,K,\xi)\cong\widehat{CFK}_{*+d}(Y,K,\widetilde J\xi),$$
where $d=A(\widetilde J\xi)-A(\xi)$.
\end{thm}

\begin{proof}
(a)  If $\phi$ is a holomorphic disk in $\Gamma_1$
connecting $\mathbf x$ to $\mathbf y$, then $\phi$ gives rise to a
holomorphic disk $\overline{\phi}$ in $\Gamma_2$ connecting
$\mathbf x$ to $\mathbf y$. Topologically, $\overline{\phi}$ is
just $-\phi$. 

The above argument implies that $\widehat{CFK}(Y,K,\xi)\cong\widehat{CFK}(Y,K,\widetilde J\xi)$ as chain complexes.

\vspace{5pt}\noindent(b) 
The isomorphism in (a) implies that $\widetilde J$ maps $\mathcal B_{Y,K}$ into $\mathcal B_{Y,K}$.

If $\xi\in\mathcal B_{Y,K}$ is represented by $\mathbf x$ in $\Gamma_1$, then $\widetilde J\xi\in\mathcal B_{Y,K}$ is represented by $\mathbf x$ in $\Gamma_2$. Using Lemma~\ref{lem:JInvariance},
 $\widetilde J^2\xi=\widetilde J_{\Gamma_2}\widetilde J\xi$ is represented by $\mathbf x$ in $\Gamma_1$.

\vspace{5pt}\noindent(c) Since $\mathfrak s$ is torsion, there exists an absolute $\mathbb Q$--grading on $\widehat{CF}(Y,\mathfrak s)$, 
hence an induced absolute $\mathbb Q$--grading on $\widehat{CFK}(Y,K,\xi)$. Since the isomorphism in (a) preserves the relative grading, there exists a rational number $d$, such that 
$$\widehat{CFK}_*(Y,K,\xi)\cong\widehat{CFK}_{*+d}(Y,K,\widetilde J\xi).$$
It is clear that the number $d$ does not depend on the choice of the Heegaard diagram, because both $\widehat{HFK}(Y,K,\xi)$ and $\widehat{HFK}(Y,K,\widetilde J\xi)$ are nontrivial absolutely graded groups.

Using (a), we get two isomorphisms which increase the Maslov grading by $d$:
$$g_1\co \widehat{CFK}(\Gamma_1,\xi)\to\widehat{CFK}(\Gamma_2,\widetilde J\xi),$$
$$g_2\co \widehat{CFK}(\Gamma_2,\xi)\to\widehat{CFK}(\Gamma_1,\widetilde J\xi).$$

Since both $\Gamma_1$ and $\Gamma_2$ represent $(Y,K)$, there is a grading preserving chain homotopy equivalence
$$f\co \widehat{CFK}(\Gamma_1,\xi)\to\widehat{CFK}(\Gamma_2,\xi).$$

Suppose $\mathbf x_1$ in $\Gamma_1$  is a generator for $\widehat{CFK}(\Gamma_1,\xi)$, 
let $\mathbf x_2=g_1(\mathbf x_1)$ in $\Gamma_2$.
Let $\mathbf y_2$  in $\Gamma_2$ be a generator which contributes to $f(\mathbf x_1)$, 
and let $\mathbf y_1=g_2(\mathbf y_2)$ in $\Gamma_1$. Since $g_1,g_2$ increase the grading by $d$ and $f$ is grading preserving, we have
\begin{equation}\label{eq:GrRel}
\mathrm{Gr}(\mathbf x_1)=\mathrm{Gr}(\mathbf y_2)=\mathrm{Gr}(\mathbf y_1)-d=\mathrm{Gr}(\mathbf x_2)-d.
\end{equation}

Since $\mathfrak s$ is torsion and $[K]$ is rationally null-homologous, $J\mathfrak s+\mathrm{PD}[K]$ is also torsion. Using Lemma~\ref{lem:Js+K}, we conclude that there exists a rational domain $\psi_1$ in $\Gamma_1$ connecting $\mathbf y_1$ to $\mathbf x_1$, such that $n_{w}(\psi_1)=0$. 
By Lemma~\ref{lem:LL} and (\ref{eq:GrRel}), we see that
\begin{equation}\label{eq:mu1}
\mu(\psi_1)=\mathrm{Gr}(\mathbf y_1)-\mathrm{Gr}(\mathbf x_1)=d. 
\end{equation}
Moreover, by Lemma~\ref{lem:RelAlex}
\begin{equation}\label{eq:n_z}
n_z(\psi_1)=A(\mathbf y_1)-A(\mathbf x_1)=A(\widetilde J\xi)-A(\xi).
\end{equation}
Noting that $\psi_2=(-\psi_1)-n_z(\psi_1)(-\Sigma)$ is a rational domain in $\Gamma_2$ that connects $\mathbf y_2$ to $\mathbf x_2$ with $n_z(\psi_2)=0$, and that $\mu(-\Sigma)=2$ in $\Gamma_2$, we have
\begin{equation}\label{eq:mu2}
\mathrm{Gr}(\mathbf y_2)-\mathrm{Gr}(\mathbf x_2)=\mu(\psi_2)=\mu(\psi_1)-2n_z(\psi_1).
\end{equation}
It follows from (\ref{eq:mu1}), (\ref{eq:mu2}) and (\ref{eq:GrRel}) that
$$\mu(\psi_1)+(\mu(\psi_1)-2n_z(\psi_1))=0.$$
Hence $$\mu(\psi_1)=n_z(\psi_1),$$
so it follows from (\ref{eq:mu1}) and (\ref{eq:n_z}) that $d=A(\widetilde J\xi)-A(\xi)$.
\end{proof}

For any $\mathfrak s\in\spin(Y)$, let
$$\widehat{HFK}(Y,K,\mathfrak s)=\bigoplus_{\xi\in\relspin(Y,K),\:G_{Y,K}(\xi)=\mathfrak s}\widehat{HFK}(Y,K,\xi).$$

\begin{cor}\label{cor:Isom}
Suppose $K$ is a rationally null-homologous knot in $Y$, $\mathfrak s$ is a Spin$^c$ structure over $Y$. Then there is an isomorphism
$$\iota\co\widehat{HFK}(Y,K,\mathfrak s)\cong\widehat{HFK}(Y,K,J\mathfrak s+\mathrm{PD}[K]).$$
If $\mathfrak s$ is torsion, and $\xi\in\relspin(Y,K)\in G_{Y,K}^{-1}(\mathfrak s)$, then the restriction of $\iota$ on $\widehat{HFK}(Y,K,\xi)$ is homogeneous of degree $A(\widetilde J\xi)-A(\xi)$.
\end{cor}
\begin{proof}
This follows from Theorem~\ref{thm:RatSymm} and Lemma~\ref{lem:Js+K}.
\end{proof}

\begin{lem}\label{lem:J}
Suppose $\xi_1,\xi_2\in\mathcal B_{Y,K}$, then
$$\widetilde J\xi_1-\widetilde J\xi_2=-(\xi_1-\xi_2)\in H^2(Y,K).$$
\end{lem}
\begin{proof}
Suppose $\xi_1,\xi_2$ are represented by intersection points $\mathbf x,\mathbf y$. Let $a$ be a multi $\alpha$ arc connecting $\mathbf y$ to $\mathbf x$, $b$ be a multi $\beta$ arc connecting $\mathbf x$ to $\mathbf y$. By \cite[Lemma~2.19]{OSzAnn1}, $\xi_1-\xi_2$ is represented by $a+b$, and $\widetilde J\xi_1-\widetilde J\xi_2$ is represented by $(-b)+(-a)$. So our conclusion holds.
\end{proof}

\begin{cor}\label{cor:MaxJMin}
Suppose $\xi\in\mathcal B_{Y,K}$, then $A(\xi)=A_{\max}$ if and only if $A(\widetilde J\xi)=A_{\min}$.
\end{cor}
\begin{proof}
If $A(\xi)\ge A(\eta)$ for all $\eta\in \mathcal B_{Y,K}$, then Lemma~\ref{lem:J} shows that $A(\widetilde J\xi)\le A(\widetilde J\eta)$ for all $\eta\in\mathcal B_{Y,K}$. Since $\widetilde J$ surjects onto $\mathcal B_{Y,K}$, $A(\widetilde J\xi)=A_{\min}$.
\end{proof}

\begin{rem}
If we choose the affine map $A$ such that $A_{\max}=-A_{\min}$, then the above corollary implies that $A(\widetilde J\xi)=-A(\xi)$.
\end{rem}


\section{A lower bound for $\Theta$}\label{sect:Bound}

In this section, we will prove Theorem~\ref{thm:GenusBound}. For simplicity, we will work over a fixed field $\mathbb F$. (A priori, the correction terms defined over different fields may not be the same, but they have similar properties. When $\mathbb F=\mathbb Q$, the correction terms are the same as the original correction terms defined over $\mathbb Z$.)

\subsection{Computing correction terms from $CFK^{\infty}$}

Fix a doubly pointed Heegaard diagram $\Gamma_1=(\Sigma,\mbox{\boldmath${\alpha}$},
\mbox{\boldmath$\beta$},w,z)$ and consider the associated knot Floer chain complex $CFK^\infty(Y,K,\xi)$ with $G_{Y,K}(\xi)=\mathfrak s$. Recall that $CFK^{\infty}(Y,K,\xi)$ is an
abelian group freely generated by triples $[\mathbf y,i,j]$ with
$$\mathbf y\in\mathbb T_{\alpha}\cap\mathbb T_{\beta}$$ and $$\underline{\mathfrak
s}_{w,z}(\mathbf y)+(i-j)PD[\mu]=\xi.$$

Let $\mathfrak G=\mathfrak{G}_{Y,K}$ be a set of generators of $\widehat{HFK}(Y,K)$, such that each generator is supported in a single relative Spin$^c$ structure and a single Maslov grading. 
By \cite[Lemma~4.5]{RasThesis}, $CFK^\infty(Y,K,\xi)$ is homotopy equivalent to a chain complex whose underlying abelian group is $\widehat{HFK}(Y,K,\xi)\otimes\mathbb F[U,U^{-1}]$, so we may assume $CFK^\infty(Y,K,\xi)$  is generated by generators $[\mathbf x,i,j]$ satisfying that every $\mathbf x$ is in $\mathfrak{G}$.

Since $Y$ is a rational homology sphere, $HF^\infty(Y,\mathfrak s) \cong \mathbb{F}[U,U^{-1}]$.
Fix a sufficiently large integer $N$.  Let $\mathcal G_{\mathfrak s}\subset CFK^\infty(Y,K,\xi)$ be the set that consists of all homogeneous chains that represent $U^{-N}\in HF^\infty(Y,\mathfrak s)$:
$$\mathcal G_{\mathfrak s}=\left\{X=\left.\sum_{\mathbf x\in \mathfrak{G}, i,j\in \mathbb Z}a_{\mathbf x,i,j}[\mathbf x,i,j]\, \right| \,  [X]=U^{-N}\in HF^\infty(Y,\mathfrak s), \mathrm{Gr}[\mathbf x,i,j]=d(Y,\mathfrak s)+2N \right\}$$
where $\mathrm{Gr}$ is the absolute Maslov grading.

\begin{lem}\label{lem:Minimax}
With the above notation, $$N= \min_{X\in \mathcal G_{\mathfrak s}}  \max _{[\mathbf x,i,j]\in X} i. $$
Here, $[\mathbf x,i,j]\in X$ means that the coefficient of $[\mathbf x,i,j]$ in the chain $X$ is nonzero.
\end{lem}

\begin{proof}
For $X\in \mathcal G_{\mathfrak s}$, let $I(X)=\max_{[\mathbf x,i,j]\in X} i$.  Then, $$U^{I(X)+1}\cdot X=\sum a_{\mathbf x,i,j}[\mathbf x,i-I(X)-1, j-I(X)-1]=0 \in HF^+(Y,{\mathfrak s})$$  since $i-I(X)-1<0$.  Hence, $N\leq I(X)$, $\forall X\in \mathcal G_{\mathfrak s}$.

On the other hand, let $X_0\in \mathcal G_{\mathfrak s}$ be a chain with $I(X_0)=\min_{X\in \mathcal G_{\mathfrak s}}I(X)$.  We claim that $$U^{I(X_0)}\cdot X_0 \neq 0 \in HF^+(Y,{\mathfrak s}),$$ which would imply $N \geq I(X_0)$.  We prove the claim by contradiction: If not, there is a $Z\in CFK^\infty(Y,K)$ of homogeneous grading $$\mathrm{Gr}(Z)=d(Y,{\mathfrak s})+2N-2I(X_0)+1$$ such that $\partial Z=U^{I(X_0)}\cdot X_0$ in the quotient complex $CFK^+(Y,K,\xi)=C\{i\geq 0\}$.  Equivalently, in $CFK^{\infty}(Y,K,\xi)$ we have $$U^{I(X_0)}\cdot X_0-\partial Z =\sum b_{\mathbf x,i,j} [\mathbf x,i,j]$$ where all $i <0$.  Let $X'=X_0-\partial (U^{-I(X_0)} Z)$.  It is clear from the construction that $X'\in \mathcal G_{\mathfrak s}$ and $I(X') < I(X_0)$.  This contradicts the assumption that $I(X_0)=\min_{X\in \mathcal G_{\mathfrak s}}I(X)$.

Therefore, we proved $N=I(X_0)=\min_{X\in \mathcal G_{\mathfrak s}}I(X).$
\end{proof}

\begin{prop}\label{prop:Correct}
With the same assumption,
\begin{equation}\label{eq:correct}
d(Y,{\mathfrak s})= \max_{X\in \mathcal G_{\mathfrak s}} \min _{[\mathbf x,i,j]\in X} \mathrm{Gr}(\mathbf x).
\end{equation}
\end{prop}

\begin{proof}
Since $X=\sum a_{\mathbf x,i,j}[\mathbf x,i,j]$ is homogeneous, we have $\mathrm{Gr}(\mathbf x)=\mathrm{Gr}(X)-2i$.  Therefore, 
\begin{eqnarray*} 
d(Y,{\mathfrak s}) &=& \mathrm{Gr}(X)-2N \\
&=& \mathrm{Gr}(X) - 2 \min_{X\in \mathcal G_{\mathfrak s}}  \max _{[\mathbf x,i,j]\in X} i\\
&=& \max_{X\in \mathcal G_{\mathfrak s}}  \min _{[\mathbf x,i,j]\in X} (\mathrm{Gr}(X)-2i) \\
&=& \max_{X\in \mathcal G_{\mathfrak s}}  \min _{[\mathbf x,i,j]\in X} \mathrm{Gr}(\mathbf x).
\end{eqnarray*}
\end{proof}

\subsection{More symmetries}

Recall that the chain complex $CFK^{\infty}(Y,K,\xi)$ can be viewed at the same time as $CFK^\infty(Y,K,\widetilde J\xi)$ associated to the Heegaard diagram $$\Gamma_2=(-\Sigma,
\mbox{\boldmath$\beta$},\mbox{\boldmath${\alpha}$},z,w).$$ There is a natural identification between intersection points in $\Gamma_1$ and $\Gamma_2$, and this can be extended to a chain isomorphism $f\co CFK^{\infty}_{\Gamma_1}(Y,K)\rightarrow CFK^{\infty}_{\Gamma_2}(Y,K)$ given by $f([\mathbf x,i,j])=[\mathbf x,j,i]$, where $\mathbf x\in \mathbb T_{\alpha} \cap \mathbb T_{\beta}$, and $i,j\in \mathbb Z$ denote the filtration with respect to $w, z$ respectively.  

\begin{lem}
Under the isomorphism $f$, the set $\mathcal G_{\mathfrak s}$ is identified with the set of all homogeneous generators that represent $U^{-M}\in HF^\infty(Y,J{\mathfrak s}+\mathrm{PD}[K])$ for some large integer $M$, associated to the Heegaard diagram $\Gamma_2$.
\end{lem}

\begin{proof}
By Lemma~\ref{lem:Js+K}, the map $$f\co CFK^{\infty}(Y,K,\xi)\rightarrow CFK^{\infty}(Y,K,\widetilde J\xi)$$ descends to $$f\co CF^{\infty}(Y,{\mathfrak s})\rightarrow CF^{\infty}(Y,J{\mathfrak s}+\mathrm{PD}[K]).$$  Moreover, since $f$ is a chain isomorphism, each element of $f(\mathcal G_{\mathfrak s})$ must represent a certain generator $U^{-M} \in HF^{\infty}(Y,J{\mathfrak s}+\mathrm{PD}[K])$ for some $M$.  

Finally, we need to prove that the elements in $f(\mathcal G_{\mathfrak s})$ are homogeneous. 
Let $\mathrm{Gr}_k$ denote the grading pertaining to the Heegaard diagram $\Gamma_k$.
 Suppose $[\mathbf x_1,i_1,j_1]$ and $[\mathbf x_2,i_2,j_2]$ contribute to $X=\sum a_{\mathbf x,i,j}[\mathbf x,i,j]\in\mathcal G_{\mathfrak s}$, then $\mathrm{Gr}_1([\mathbf x_1,i_1,j_1])=\mathrm{Gr}_1([\mathbf x_2,i_2,j_2])$.    Since $\mathbf x_1$ and $\mathbf x_2$ belong to the same Spin$^c$ structure, there exists a topological disk $\phi$ in $\Gamma_1$ connecting them.  By adding an appropriate multiple of $\Sigma$, we may further assume that $n_w(\phi)=i_1-i_2$.  Thus, $\mu(\phi)=0$ according to the Maslov index formula.  Moreover, since $[\mathbf x_1,i_1,j_1]$ and $[\mathbf x_2,i_2,j_2]$ satisfy (\ref{eq:Generator}), we have 
 $$\underline{\mathfrak
s}^1_{w,z}(\mathbf x_1)+(i_1-j_1)PD[\mu]=\underline{\mathfrak
s}^1_{w,z}(\mathbf x_2)+(i_2-j_2)PD[\mu],$$
which implies that (as in Lemma~\ref{lem:RelAlex})
$$n_z(\phi)-n_w(\phi)+(i_1-j_1)=(i_2-j_2).$$
So
 $n_z(\phi)=j_1-j_2$.  Consequently, the disk $-\phi$ connects $[\mathbf x_1,i_1,j_1]$ and $[\mathbf x_2,i_2,j_2]$ in the Heegaard diagram $\Gamma_2$.  As the Maslov index of $\phi$ is invariant under the reversion of orientation, $$\mathrm{Gr}_2([\mathbf x_1,j_1,i_1])-\mathrm{Gr}_2([\mathbf x_2,j_2,i_2])=\mu(\phi)=0.$$ This proved the elements in $f(\mathcal G_{\mathfrak s})$ are homogeneous in $\mathrm{Gr}_2$.
\end{proof}

Applying Lemma~\ref{lem:Minimax} to $f(\mathcal G_{\mathfrak s})$, the set of homogeneous generators that represent $U^{-M}\in HF^\infty(Y,J{\mathfrak s}+\mathrm{PD}[K])$, we conclude:
$$M=\min_{X\in \mathcal G_{\mathfrak s}}  \max _{[\mathbf x,i,j]\in X} j. $$
With the same argument as in Proposition~\ref{prop:Correct}, this leads to the following analogous correction term formula.
\begin{equation}\label{eq:correct2}
d(Y,J{\mathfrak s}+\mathrm{PD}[K])= \max_{X\in \mathcal G_{\mathfrak s}} \min _{[\mathbf x,i,j]\in X} \mathrm{Gr}_2(\mathbf x).
\end{equation}

\subsection{Proof of Theorem~\ref{thm:GenusBound}}
Our proof is based on the following elementary principle.

\begin{lem}\label{LocGlob}
For any bounded sequence of pairs $(a_i,b_i)\in (a,b)+\mathbb  Z^2$, where $(a,b)\in\mathbb R^2$, we have 
$$|\min_ia_i-\min_ib_i| \leq \max_i|a_i-b_i|, $$
$$|\max_ia_i-\max_ib_i| \leq \max_i|a_i-b_i|.$$
\end{lem}

\begin{proof}
The condition that $(a_i,b_i)\in (a,b)+\mathbb  Z^2$ is bounded allows us to take minimum and maximum.
Assume $a=a_m=\min_ia_i$ and $b=b_k=\min_ib_i$.  Then, $$a-b=a_m-b_k\leq a_k-b_k;$$ and $$b-a=b_k-a_m\leq b_m-a_m.$$  It readily follow that $$|a-b|\leq \max_i|a_i-b_i|.$$  The second inequality follows from the first by replacing $a_i,b_i$ with $-a_i,-b_i$.
\end{proof}

To bound $|d(Y,{\mathfrak s})-d(Y,J{\mathfrak s}+\mathrm{PD}[K])|$, we apply Lemma~\ref{LocGlob} twice to the equations (\ref{eq:correct}) and (\ref{eq:correct2}).  In the first round, let the pair $$(a_X,b_X)=\Big(\min _{[\mathbf x,i,j]\in X} \mathrm{Gr}_1(\mathbf x), \min _{[\mathbf x,i,j]\in X} \mathrm{Gr}_2(\mathbf x)\Big)$$ and $X\in \mathcal G_{\mathfrak s}$ be the index of the sequence.  We get
$$\Big|d(Y,{\mathfrak s})-d(Y,J{\mathfrak s}+\mathrm{PD}[K])\Big| \leq \max_{X\in \mathcal G_{\mathfrak s}} \Big| \min _{[\mathbf x,i,j]\in X} \mathrm{Gr}_1(\mathbf x) - \min _{[\mathbf x,i,j]\in X} \mathrm{Gr}_2(\mathbf x) \Big|. $$
In the second round, let the pair $$(a_{\mathbf x},b_{\mathbf x})=\left(\mathrm{Gr}_1(\mathbf x),\mathrm{Gr}_2(\mathbf x)\right)$$ and $\mathbf x \in \mathfrak G$ be the index of the sequence.  We get
$$\Big|\min _{[\mathbf x,i,j]\in X} \mathrm{Gr}_1(\mathbf x) - \min _{[\mathbf x,i,j]\in X} \mathrm{Gr}_2(\mathbf x) \Big| \leq \max_{[\mathbf x,i,j]\in X} \Big|\mathrm{Gr}_1(\mathbf x)-\mathrm{Gr}_2(\mathbf x)\Big|.$$
Plugging the second inequality to the first, we obtain:
\begin{equation}\label{eq:GrDiff}
|d(Y,{\mathfrak s})-d(Y,J{\mathfrak s}+\mathrm{PD}[K])| \leq \max_{\mathbf x\in\mathfrak G}|\mathrm{Gr}_1(\mathbf x)-\mathrm{Gr}_2(\mathbf x)|.
\end{equation}

The proof of Theorem~\ref{thm:RatSymm}~(b) implies that
$$\mathrm{Gr}_2(\mathbf x)-\mathrm{Gr}_1(\mathbf x)=A(\widetilde J\underline{\mathfrak s}_{w,z}(\mathbf x))-A(\underline{\mathfrak s}_{w,z}(\mathbf x)).$$
Recall that at the beginning of this section, we assumed that $\widehat{HFK}(Y, K)$ was generated by $\mathbf x\in\mathfrak G$,
Theorem~\ref{thm:RatGenus} then implies that the right hand side of (\ref{eq:GrDiff}) is bounded from above by the left hand side of (\ref{eq:GenusBound}). 
So (\ref{eq:GrDiff}) implies
$$|d(Y,J{\mathfrak s}+\mathrm{PD}[K])-d(Y,{\mathfrak s})|\le 1+\frac{-\chi(F)}{|[\partial F]\cdot [\mu]|}.$$ 
By (\ref{eq:CorrSymm}), $d(Y,\mathfrak s)=d(Y,J\mathfrak s)$. So we get
\begin{eqnarray*}
 1+\frac{-\chi(F)}{|[\partial F]\cdot [\mu]|}&\ge&\max_{\mathfrak s\in\spin(Y)}\big\{d(Y,J\mathfrak s+\mathrm{PD}[K])-d(Y,J\mathfrak s)\big\}\\
&=&\max_{\mathfrak s\in\spin(Y)}\big\{d(Y,\mathfrak s+\mathrm{PD}[K])-d(Y,\mathfrak s)\big\}.
\end{eqnarray*}
This finishes the proof of Theorem~\ref{thm:GenusBound}.


\section{Applications}\label{sect:Appl}

In this section, we apply Theorem~\ref{thm:GenusBound} to compute $\Theta$ for certain homology classes in two types of manifolds: L-spaces and large surgeries on knots in $S^3$.

\subsection{Floer simple knots in L-spaces}

\begin{prop}\label{prop:SimpleGenus}
Suppose $Y$ is an L-space, $K$ is a Floer simple knot in $Y$, $F$ is a genus minimizing rational Seifert surface for $K$. Then the Euler characteristic of $F$ is determined by the formula  
\begin{equation}\label{eq:SimpleGenus}
1+\frac{-\chi(F)}{|[\partial F]\cdot[\mu]|}=\max_{\mathfrak s\in\spin(Y)}\big\{d(Y,\mathfrak s+\mathrm{PD}[K])-d(Y,\mathfrak s)\big\}.
\end{equation}
\end{prop}
\begin{proof}
As in the proof of  Theorem~\ref{thm:GenusBound}, the right hand side of (\ref{eq:SimpleGenus}) is equal to 
$$\max_{\mathfrak s\in\spin(Y)}\big\{d(Y,J\mathfrak s+\mathrm{PD}[K])-d(Y,\mathfrak s)\big\}.$$
As $Y$ is an L-space and $K$ is Floer simple, for any $\mathfrak s\in\spin(Y)$, there exists a $\xi_{\mathfrak s}\in \relspin(Y,K)$ such that $$\widehat{HFK}(Y,K,\xi_{\mathfrak s})\cong\widehat{HFK}(Y,K,\mathfrak s)\cong\widehat{HF}(Y,\mathfrak s)\cong\mathbb Z.$$
Corollary~\ref{cor:Isom} implies that $\widetilde J\xi_{\mathfrak s}=\xi_{J\mathfrak s+\mathrm{PD}[K]}$, and $$A(\widetilde J\xi_{\mathfrak s})-A(\xi_{\mathfrak s})=d(Y,J\mathfrak s+\mathrm{PD}[K])-d(Y,\mathfrak s).$$
Since $\widehat{HFK}(Y,K)$ is supported in these $\xi_{\mathfrak s}$'s, our conclusion follows from Theorem~\ref{thm:RatGenus} and Corollary~\ref{cor:MaxJMin}.
\end{proof}

\begin{proof}[Proof of Theorem~\ref{thm:SimpleL}]
This follows from Theorem~\ref{thm:GenusBound} and Proposition~\ref{prop:SimpleGenus}.
\end{proof}

\begin{prop}\label{prop:SimpleFiber}
Suppose $K$ is a Floer simple knot in an L-space $Y$. Then $K$ is a rationally fibered knot if and only if the right hand side of (\ref{eq:SimpleGenus}) is achieved by exactly one $\mathfrak s\in\spin(Y)$.
\end{prop}
\begin{proof}
This follows from Theorem~\ref{thm:RatFiber}, Corollary~\ref{cor:MaxJMin} and the proof of  Proposition~\ref{prop:SimpleGenus}.
\end{proof}

\begin{cor}
Suppose $K_1,K_2$ are two Floer simple knots in an L-space $Y$ with $[K_1]=[K_2]\in H_1(Y;\mathbb Z)$, then $K_1$ and $K_2$ have the same rational genus, and $K_1$ is rationally fibered if and only if $K_2$ is rationally fibered.
\end{cor}
\begin{proof}
This follows from Propositions~\ref{prop:SimpleGenus} and \ref{prop:SimpleFiber} by observing that the right hand side of (\ref{eq:SimpleGenus}) only depends on the homology class $[K]$.
\end{proof}

\subsection{Large surgeries on knots}\label{subsect:LargeSurg}

In this subsection, we will consider another case of the rational genus bound.
Suppose that $K$ is a knot in a homology sphere $Y$. Let $Y_p(K)$ be the manifold obtained by $p$--surgery on $K$, and let $K'\subset Y_p(K)$ be the dual knot of the surgery. We can isotope $K'$ to be a curve on $\partial\nu(K')=\partial\nu(K)$ such that this curve is isotopic to the meridian $\mu$ of $K$. We always orient $K'$ such that the orientation coincides with the standard orientation on $\mu$. If $F$ is a Seifert surface for $K$, then $F$ (or $-F$ if one cares about the orientation) is also a rational Seifert surface for $K'$. So $\frac{2g(K)-1}p$ is an upper bound for $\Theta([K'])$. Theorem~\ref{thm:GenusBound} gives a lower bound for $\Theta([K'])$, which we will compute.

The set of Spin$^c$ structures $\spin(Y_p(K))$ is in one-to-one correspondence with $H^2(Y_p(K))\cong\mathbb Z/p\mathbb Z$.
However, this correspondence is generally not canonical. 
Ozsv\'ath and Szab\'o \cite{OSzKnot} specified an identification of $\spin(Y_p(K))$ with $\mathbb Z/p\mathbb Z$ as follows.
Let $F\subset Y$ be a Seifert surface for $K$, $W\co Y_p(K)\to Y$ be the $2$--handle cobordism, and $\widehat{F}\subset W$ be the surface obtained from $F$ by capping off $\partial F$ with the cocore of the $2$--handle. For any $i\in\mathbb Z$, let $\mathfrak x_i\in\spin(W)$ be the Spin$^c$ structure satisfying that
$$\langle c_1(\mathfrak x_i),[\widehat F]\rangle=2i-p.$$
Now we define a map $\sigma\co\spin(Y_p(K))\to\mathbb Z/p\mathbb Z$ by
$$\sigma(\mathfrak x_i|Y_p(K))\equiv i\pmod p.$$
This map is well-defined, and is the identification we want.

By \cite{OSz4Manifold}, the Spin$^c$ cobordism $(W,\mathfrak x_i)\co (Y_p(K),i)\to Y$ induces an isomorphism
$$F^{\infty}_{(W,\mathfrak x_i)}\co HF^{\infty}(Y_p(K),i)\to HF^{\infty}(Y),$$
which shifts the grading by $\frac{-(2i-p)^2+p}{4p}$. Using the definition of correction terms, we see that 
\begin{equation}\label{eq:CorrMod}
d(Y)-d(Y_{p}(K),i)\equiv\frac{-(2i-p)^2+p}{4p}\pmod2.
\end{equation}
When $Y$ is an L-space, there is a more precise formula relating $d(Y)$ and $d(Y_{p}(K),i)$ in \cite{NiWu}, which we briefly describe below. 
From $CFK^{\infty}(Y,K)$, one can define two sequences of nonnegative integers $V_k,H_k$, $k\in\mathbb Z$ satisfying that
\begin{equation}\label{eq:VH}
V_k=H_{-k},\quad V_k\ge V_{k+1}\ge V_k-1,\quad V_{g(K)}=0.
\end{equation}
When $Y$ is an L-space,
the correction terms of $Y_{p}(K)$ can be computed by the formula
\begin{equation}\label{eq:CorrSurg}
d(Y_{p}(K),i)=d(Y)+d(L(p,1),i)-2\max\{V_i,H_{i-p}\}.
\end{equation}

From \cite{OSzAbGr} we know that
\begin{equation}\label{eq:Lp1}
d(L(p,1),i)=\frac{(2i-p)^2-p}{4p}
\end{equation}
when $0\le i\le p$. Using (\ref{eq:CorrSurg}), when
 $0\le i<p$, we get
\begin{eqnarray*}
&&d(Y_p(K),i+1)-d(Y_p(K),i)\\
&=&\frac{(2i+2-p)^2-p}{4p}-2\max\{V_{i+1},H_{i+1-p}\}-\left(\frac{(2i-p)^2-p}{4p}-2\max\{V_i,H_{i-p}\}\right) \\
&=&\frac{2i+1-p}p-2\max\{V_{i+1},V_{p-1-i}\}+2\max\{V_i,V_{p-i}\}.
\end{eqnarray*}

Applying Theorem~\ref{thm:GenusBound}, we see that $\Theta([K'])$ is bounded from below by
\begin{equation}\label{eq:SurgBound}
\max_{i\in\{0,1,\dots,p-1\}}\left\{\frac{2i+1-2p}p-2\max\{V_{i+1},V_{p-1-i}\}+2\max\{V_i,V_{p-i}\}\right\}.
\end{equation}

The bound given by (\ref{eq:SurgBound}) is not always sharp, as there are nontrivial knots with $V_k=0$ whenever $k\ge0$. In this case the result of (\ref{eq:SurgBound}) is $-\frac1p$. However, we can still compute $\Theta([K'])$ for large surgeries on knots in $S^3$.

\begin{lem}\label{lem:SpinId}
Suppose that $Y,Z$ are two homology spheres, and $K\subset Y$ and $L\subset Z$ are two knots. Let $p$ be a positive integer, $K'\subset Y_p(K)$ and $L'\subset Z_p(L)$ be the dual knots of the surgeries. If there is an orientation preserving homeomorphism $f\co Y_p(K)\to Z_p(L)$ with $f_*[K']=[L']$, then the induced map on the Spin$^c$ structures $$f_{\star}\co\spin(Y_p(K))\to \spin(Z_p(L))$$ is given by the identity map of $\mathbb Z/p\mathbb Z$.
\end{lem}
\begin{proof}
From the identification $\spin(Y_p(K))\cong\mathbb Z/p\mathbb Z$ described before, we can conclude that the conjugation $J$ on $\spin(Y_p(K))$ is given by $$J(i)\equiv -i\pmod p.$$ Since $f$ is a homeomorphism, one should have 
\begin{equation}\label{eq:CommJ}
f_{\star}J=Jf_{\star}.
\end{equation}
Moreover, the homology class $[K']$ corresponds to $1\in H_1(Y_p(K))\cong\mathbb Z/p\mathbb Z$, and a similar result is true for $[L']$. Since $f_*[K']=[L']$, $f_{\star}$ should satisfy 
\begin{equation}\label{eq:+1}
f_{\star}(i+1)-f_{\star}(i)=1.
\end{equation}
When $p$ is odd, the only affine isomorphism on $\mathbb Z/p\mathbb Z$ satisfying (\ref{eq:CommJ}) and (\ref{eq:+1}) is the identity, so our conclusion holds in this case.

When $p=2n$ is even, the affine isomorphisms satisfying (\ref{eq:CommJ}) and (\ref{eq:+1}) are the identity and $i\mapsto i+\frac{p}2$, we only need to show that the latter case cannot happen. Otherwise, we should have
$$d(Y_p(K),i)=d(Z_p(L),i+n),\quad i\in\mathbb Z/(2n\mathbb Z).$$
By (\ref{eq:CorrMod}), we get
$$d(Y)+\frac{2(i-n)^2-n}{4n}\equiv d(Z)+\frac{2i^2-n}{4n} \pmod2,\quad \text{when }0\le i\le n.$$
So
\begin{equation}\label{eq:CorrDiff}
d(Y)-d(Z)\equiv i-\frac{n}2\pmod2,\quad \text{when }0\le i\le n.
\end{equation}
Noting that the correction term of a homology sphere is always an even integer,
so the right hand side of (\ref{eq:CorrDiff}) is even for any $i=0,\dots,n$, which is impossible.
\end{proof}

\begin{proof}[Proof of Proposition~\ref{prop:S3Surg}]
If $K'$ is not genus minimizing, then there exists a knot $L'\subset S^3_p(K)$ with $[L']=[K']$ and $||L'||<||K'||$. There is a natural isomorphism $H_1(\partial\nu(L'))\cong H_1(\partial\nu(K'))$. Let $\mu_L$ be the slope on $L'$ corresponding to the meridian of $K$ under the previous isomorphism. Let $Z$ be the manifold obtained from $S^3_p(K)$ by $\mu_L$--surgery on $L'$, and let $L$ be the dual knot. Then it is elementary to check that $Z$ is a homology sphere, $S^3_p(K)=Z_p(L)$ and $g(L)<g(K)$.

By Lemma~\ref{lem:SpinId}, we have
\begin{equation}\label{eq:HF=}
HF^+(S^3_p(K),i)\cong HF^+(Z_p(L),i), \quad\text{for any }i\in\mathbb Z/p\mathbb Z.
\end{equation}
Since $p>2g(K)-1>2g(L)-1$,  it follows from \cite[Theorem~4.4]{OSzKnot} that  $$HF^+(S^3_p(K),g(K))\cong HF^+(S^3),\quad HF^+(Z_p(L),g(K))\cong HF^+(Z).$$
By (\ref{eq:HF=}), we have $HF^+(Z)\cong HF^+(S^3)$, hence $Z$ is an L-space.

Since $p>2g(L)-1$ and $g(K)-1\ge g(L)$, we have
\begin{equation}\label{eq:Z}
HF^+(S^3_p(K),g(K)-1)\cong HF^+(Z_p(L),g(K)-1)\cong HF^+(Z).
\end{equation}
 For $C=CFK^{\infty}(S^3,K)$, consider the natural short exact sequence
$$
\begin{CD}
0@>>>C\{i<0,j\ge g(K)-1\}@>>>C\{i\ge0\text{ or }j\ge g(K)-1\}@>>>C\{i\ge0\}@>>>0,
\end{CD}
$$
which induces a long exact sequence
$$
\begin{CD}
\cdots@>>> \widehat{HFK}(K,g(K))@>>> HF^+(S^3_p(K),g(K)-1)@>v_{g(K)-1}>>HF^+(S^3)@>>>\cdots.
\end{CD}
$$
We have $HF^+(S^3)\cong\mathcal T^+:=\mathbb Z[U,U^{-1}]/U\mathbb Z[U]$. By (\ref{eq:Z}), $$HF^+(S^3_p(K),g(K)-1)\cong HF^+(Z)\cong\mathcal T^+.$$
Hence $v_{g(K)-1}$ is equivalent to $U^{V_{g(K)-1}}\co \mathcal T^+\to\mathcal T^+$. As $\widehat{HFK}(K,g(K))\ne 0$ \cite{OSzGenus}, we have $V_{g(K)-1}>0$. By (\ref{eq:VH}), $V_{g(K)-1}=1$ and $V_{g(K)}=0$.

In (\ref{eq:SurgBound}), letting $i=g(K)-1$ and using the fact that $p\ge 2g(K)$, we get
$$\Theta([K'])\ge\frac{2g(K)-2+1-2p}p-2V_{g(K)}+2V_{g(K)-1}=\frac{2g(K)-1}p,$$
which contradicts the assumption that $K'$ is not genus minimizing.
\end{proof}

\end{document}